\documentclass[amsart,11pt]{article}
\usepackage{authblk}

\usepackage{amsthm}[11pt]
\usepackage{bbm}

\usepackage{color,graphicx,latexsym,amssymb,amsmath,amsfonts,amsthm,float}
\usepackage{url}

\numberwithin{equation}{section}

\def\qed{ \hfill $\Box$}

\makeatletter

\renewcommand{\section}{
  \@startsection
  {section}
  {1}
  {0pt}
  {1.1\baselineskip}
  {0.2\baselineskip}
  {\sc \centering}
}
\makeatother

\newcommand{\QCLT}{QIP }
\newcommand{\QIP}{QIP}
\newcommand{\Pomega}{\mathbb{P}^{\omega}}
\newcommand{\Eomega}{\mathbb{E}^{\omega}}

\newenvironment{sketchpfoflem}[1]
{\par\vskip2\parsep\noindent{\sc Sketch of Proof of Lemma\ #1.}}{{\hfill
$\Box$}
\par\vskip2\parsep}

\newenvironment{pfofthm}[1]
{\par\vskip2\parsep\noindent{\sc Proof of Theorem\ #1.}}{{\hfill
$\Box$}
\par\vskip2\parsep}

\newcommand{\Z}{\mathbb Z}

\newcommand{\Omegahat}{\hat \Omega}
\newcommand{\that}{\hat T}

\newcommand{\hatP}{\hat{\mathbb{P}}}

\newcommand{\T}{T}

\definecolor{db}{rgb}{0.1,0,0.75}
\definecolor{lm}{cmyk}{0 ,1,0,0}

\newcommand{\pr}{\mathbb P}
\newcommand{\N}{\mathbb N}

\newtheorem{lem}{Lemma}

\newtheorem{thm}{Theorem}

\newtheorem{cor}{Corollary}

\DeclareMathOperator{\sgn}{sgn}

\begin{document}
\def\shorttitle{\textbf{Random Mass Splitting}}
\title{\textbf{\Large \sc Random mass splitting and a quenched invariance principle}}

\author[1]{Sayan Banerjee\thanks{sayan.banerjee20@gmail.com}}
\author[2]{Christopher Hoffman\thanks{hoffman@math.washington.edu}}
 \affil[1]{University of Warwick}
 \affil[2]{University of Washington}
\date{\today}
\maketitle
\begin{abstract}
We will investigate a random mass splitting model and the closely related random walk in a random environment (RWRE). The heat kernel for the RWRE at time $t$ is the mass splitting distribution at $t$. We prove a quenched invariance principle (\QIP) for the RWRE which gives us
a quenched central limit theorem for the mass splitting model. Our RWRE has an environment which is changing with time. We follow the outline for proving a \QIP{} for a random walk in a space-time random environment laid out by Rassoul-Agha and Sepp{\"a}l{\"a}inen \cite{timo} which in turn was based on the work of Kipnis and Varadhan  \cite{kipnis} and others.
\end{abstract}

\section{Introduction}
Imagine a one dimensional city (motivated by Abbott's \textit{Flatland}) on the $Y$-axis with houses at the points $(0,k)$. Suppose the city is stricken with an epidemic and things are getting worse by the day. As the death toll rises, each house $(0,k)$ has only a Poisson$(1)$ number of survivors $v(0,k)$ when the long awaited medical breakthrough suddenly happens. The surviving residents (assume at least one survives) of house $(0,0)$ are scientists who were working on a cure for a while and they finally have an antidote! But the quantity that is produced is limited (say has mass $1$ unit). So, they get out of their house carrying an equal proportion of the medicine $\displaystyle{\left(1/v(0,0)\right)}$, and start performing one-dimensional simple random walks, with time represented along the positive $X$-axis and the spatial coordinate along the $Y$-axis, to share it with the survivors. Meanwhile, the remaining survivors decide that staying inside the house in unhygienic conditions is dangerous, and they get out of their respective houses and start performing (one-dimensional) simple random walks at the same time as the scientists. Whenever a group of people meet on the way, say at $(t,k)$, and at least one person in the group has some medicine, it gets equally divided among all the people in the group. This process continues. The question we ask in this article is: How does the mass distribution of medicine at time $t$ behave for large values of $t$?

We call this system the \textit{Random Mass Splitting} model. This model came about as the natural discrete analogue of the following problem which was suggested to us by Krzysztof Burdzy \cite{burdzy}: suppose there is an initial configuration of particles distributed according to a Poisson point process. Each particle performs a Brownian motion when away from its two neighboring particles, and when two particles meet, they get reflected off each other so that the \textit{ordering is preserved}. This is the two-sided version of the \textit{Atlas model} (see \cite{pal}) which frequently appears in Stochastic Portfolio theory. We give a mass of one to a tagged particle. Every time two particles collide, the mass splits in half between them (it can be shown that there are no triple collisions by methods of \cite{pal}). A way of understanding the intersection graph of these particles is to monitor the evolution of this mass in time. We hope that the developments of this paper will shed light on this problem.

Our model fits in a well-studied class of models where two types of particles interact. These problems become substantially harder when \textit{both types of particles move}. 
In \cite{kesten2005} \cite{kesten2006} and \cite{kesten2008} Kesten and Sidoravicius investigated the shape of the infected set when a group of moving particles spread a contagious disease (or rumour) among another group of healthy (ignorant) particles which also move. The movement of each particle is assumed to be independent of the other. Another example from this class is the work of Peres et.\ al.\ \cite{peres} who studied the  detection of a moving particle by a mobile wireless network. Our model is more in the spirit of the former collection of papers. However instead of looking at a shape theorem for the sites where the antidote has reached, we are interested in \textit{quantifying the spread} by looking at the distribution of mass that is initially carried by a few particles (the ones at the origin). 

Our study of the model begins with the following observation. We use the movement of the villagers over time to define a random environment. Then we study the movement of a random walker in this random environment (RWRE).  We study this RWRE because it turns out that the heat kernel of the walk conditioned on the environment $\omega$ at time $t$ is precisely the distribution of the medicine at time $t$ when the villagers are moving according to $\omega$. Random walk in random environment models have been studied by many authors. This can be a very difficult field as even the simplest properties such as transience and recurrence are difficult to establish \cite{kalikow}. But a theory of random walk in random environments has been developed. We will harness that theory to find the asymptotic mass distribution.

In many examples, such as simple random walk on supercritical percolation clusters, we get a quenched invariance principle (i.e.\ an almost sure convergence of the RWRE to Brownian motion) \cite{begerbiskup} \cite{ss}. However in other examples we find behaviors that are very different from the usual diffusive behavior of simple random walk on $\Z^d$
\cite{sinai} \cite{bbhk}.
%
%
Many of the proofs of invariance principles for random walks in random environments 
have their origins in the seminal work of Kipnis and Varadhan \cite{kipnis}. This paper laid down the foundation for quenched invariance principles for reversible Markov chains. Maxwell and Woodroofe \cite{maxwellwoodroofe} and Derriennic and Lin \cite{lin} subsequently extended their approach to the non-reversible set-up. Rassoul-Agha and Sepp{\"a}l{\"a}inen \cite{timo} developed further on these techniques to give a set of conditions under which a quenched invariance principle (\QIP) holds for random walks in space-time random environments. They verified these conditions to prove a \QCLT for \textit{i.i.d. environments} (see \cite{timo}). 
A different method for proving a quenched CLT was developed by Bolthausen and Sznitman  \cite{sznit} and refined by Berger and Zeitouni \cite{bergerquenched}.

In this article, we broadly follow the approach of \cite{timo}. Verifying their conditions for a \QCLT presents significantly more complications in our situation than it does with their i.i.d.\ environments. Especially, the moment condition (see Subsection \ref{subsection:momentcondition}) presents quite a hurdle. Also, the absence of \textit{ellipticity} in the model (to be explained later) requires us to use some abstract concepts from ergodic theory, like `Kakutani skyscrapers' (see Subsection \ref{subsection:ergodicity}).
That these complications can be overcome demonstrates the robustness of the approach in \cite{timo}.

\section{Description of the model}\label{section:description}
We now present a more rigorous mathematical description of our model. Consider i.i.d.\ Poisson$(1)$ random variables $\{v(0,k): k \in \mathbb{Z}\}$ representing particles at points $(0,k)$ of the $Y$-axis at time $t=0$. Each particle starts performing a one-dimensional simple random walk $\{S^{(k,i)}(\cdot): 1 \le i \le v(0,k), k \in \mathbb{Z}\}$, with $S^{(k,i)}(0)=k$ and time represented along the positive $X$-axis. This gives a collection of random variables $\{e^+(t,y),e^-(t,y), v(t,y)\}$ where
\begin{eqnarray*}
e^+(t,y)= \#\{(k,i): S^{(k,i)}(t)=y, S^{(k,i)}(t+1)=y+1\},\\\\
e^-(t,y)= \#\{(k,i): S^{(k,i)}(t)=y, S^{(k,i)}(t+1)=y-1\},\\\\
v(t,y)= \#\{(k,i): S^{(k,i)}(t)=y\}= e^+(t,y)+ e^-(t,y).
\end{eqnarray*}
Note that the following also holds: $$v(t,y)=e^+(t-1,y-1)+e^-(t-1,y+1).$$ 
For technical convenience, we take the $S^{(k,i)}$ to be two-sided random walks (i.e. defined for both positive and negative times) and thus the above random variables are defined for all $(t,y) \in \mathbb{Z}^2$.\\\\

Now we mathematically describe the mass distribution of the medicine. For $t \ge 0$, $p(t,y)$ denote the (random amount of) mass at $(t,y)$. 
Then, for each $t$, $$\displaystyle{\sum_{y=-t}^{t}p(t,y)=1}.$$ Thus, $p(0,0)=1$. At time $1$, assuming $v(0,0)>0$, we have $$\displaystyle{p(1,1)=\frac{e^+(0,0)}{v(0,0)}}\ \ \ \ \text{and} \ \ \ \ \displaystyle{p(1,-1)=\frac{e^-(0,0)}{v(0,0)}}.$$ Now, suppose $e^+(0,0)>0$. After the mass gets split equally among the particles at $(1,1)$, each one has mass $\displaystyle{\frac{e^+(0,0)}{v(0,0)v(1,1)}}$. A similar statement holds if $e^-(0,0)>0$. Thus at time $2$, the mass distribution is 
\begin{eqnarray*}
p(2,2)&=&\frac{e^+(0,0)e^+(1,1)}{v(0,0)v(1,1)},\\\\
p(2,0)&=&\frac{e^+(0,0)e^-(1,1)}{v(0,0)v(1,1)}+\frac{e^-(0,0)e^+(1,-1)}{v(0,0)v(1,-1)} \ \ \mbox{ and }\\\\
p(2,-2)&=&\frac{e^-(0,0)e^-(1,-1)}{v(0,0)v(1,-1)}.
\end{eqnarray*}
Here we make a convention of taking any term in the above expression whose denominator (and hence numerator) vanishes to be equal to zero.

To write down the explicit distribution of mass at time $t$, we define the \textit{environment} to be the collection of `random edge crossings'
\begin{equation*}
\omega=\left\{\left(e^+(t,y),e^-(t,y)\right): (t,y) \in \mathbb{Z}^2\right\}
 \in ((\mathbb{Z}_+\cup\{0\})^2)^{\mathbb{Z}^2}.
\end{equation*}
We call the space of environments $\Omega =((\mathbb{Z}_+\cup\{0\})^2)^{\mathbb{Z}^2}=\{(a(t,y),b(t,y))\in (\mathbb{Z}_+\cup\{0\})^2: (t,y) \in \mathbb{Z}^2\}$. 
With these definitions we can inductively define the mass distribution with the equation
\begin{equation} \label{simplifies}
p(t+1,y;\omega)=p(t,y+1;\omega)\frac{e^-(t,y+1)}{v(t,y+1)}
     +  p(t,y-1;\omega)\frac{e^+(t,y-1)}{v(t,y-1)}.
\end{equation}

The space $\Omega$ is equipped with the canonical product sigma field $\mathcal{G}$ and the natural shift transformation $T_{(t',y')}$ such that
\begin{eqnarray*}
\left(a(t,y),b(t,y)\right)\left(T_{(t',y')}\omega\right)=\left(a(t+t',y+y'),b(t+t',y+y')\right)(\omega).
\end{eqnarray*}
The random edge crossings induce a measure $\mathbb{P}$ on $(\Omega,\mathcal{G})$ and it is easy to see that $\mathbb{P}$ is invariant under the shift $T_{(t',y')}$. It can be further proved that $\left(\Omega,\mathcal{G},\left\lbrace T_{(t',y')}\right\rbrace_{(t',y') \in \mathbb{Z}^2},\mathbb{P}\right)$ is ergodic. We leave this verification to the reader. We shall denote the corresponding expectation by $\mathbb{E}$.\\\\

Now fix an $\omega$. 
This generates a random mass distribution at $(t,y)$ which we write as $p(t,y;\omega)$. 
Let $\Gamma_t(0,y)$ denote the set of all nearest neighbor paths $\gamma: \{0,1,\dots,t\} \rightarrow \mathbb{Z}$ with $\gamma(0)=0$ and $\gamma(t)=y$. Then the mass at $(t,y)$ is given by
\begin{eqnarray*}
p(t,y;\omega)=\sum_{\gamma \in \Gamma_t(0,y)}\prod_{i=0}^{t-1}\frac{e^{\sgn(\gamma(i+1)-\gamma(i))}(i,\gamma(i))}{v(i,\gamma(i))}\mathbbm{1}(v(i,\gamma(i))>0),
\end{eqnarray*}
where a product in the above sum is interpreted as zero if one of the terms in the product vanishes. 

Note that the mass remains $1$ at $(0,0)$ at all times when $v(0,0)=0$. To avoid this, we make a change of measure to $\tilde{\mathbb{P}}$ where $$\tilde{\mathbb{P}}(A)=\mathbb{E}(v(0,0)\mathbbm{1}_A)$$ for $A \in \mathcal{G}$. Note that $\tilde{\mathbb{P}}$ is a probability measure as $\mathbb{E}(v(0,0))=1$. This change of measure ensures $\tilde{\mathbb{P}}(v(0,0)>0)=1$. So from now on we only consider $\omega$ with $v(0,0)>0$. 

Given any $\omega \in \Omega$ and any $(t,y) \in \Z^2$, we will now describe a random walk in the random environment $\omega$ started at space point $y$ and time point $t$. The law of this random walk will be given in accordance with the way the mass at $(t,y)$ propagates in the future. Once this link is made precise (see Lemma \ref{equivalence}), the results we prove about this random walk will automatically translate to results about the random mass distribution $p$. 

We write $X_n(t,y;\omega)$ (or simply $X_n$ when the environment $\omega \in \Omega$ and the starting point $(t,y) \in \Z^2$ is clear from the context) for the random variable that represents the location of the random walk after $n$ steps (at time $t+n$). The distribution of $(X_n(t,y;\omega))_{n \ge 0}$, which we denote by $\Pomega_{(t,y)}$, is defined as follows.
\begin{eqnarray*}
\Pomega_{(t,y)}(X_0=(t,y))=1,
\end{eqnarray*}
and $(X_n)_{n \ge 0}$ is a time inhomogeneous Markov chain with transition probabilities
\begin{eqnarray} \label{ravegreen}
\Pomega_{(t,y)}(X_{n+1}=(t'+1,y'+1)|X_n=(t',y'))&=&\frac{e^+(t',y')}{v(t',y')} \ \ \ \  \text{and} \\
\Pomega_{(t,y)}(X_{n+1}=(t'+1,y'-1)|X_n=(t',y'))&=&\frac{e^-(t',y')}{v(t',y')}. \label{ravegreen2}
\end{eqnarray}

Now we present the key observation that connects the random walk in a random environment with the random mass splitting model.

\begin{lem} \label{equivalence}
For $t \in \N$ and $y \in \Z$ 
$$p(t,y;\omega)=\pr_{(0,0)}^{\omega}(X_t=(t,y)).$$
\end{lem}

\begin{proof}
This follows by induction on $t$. Fix an environment $\omega$.  Then we have
$$p(0,0;\omega)=\pr_{(0,0)}^{\omega}(X_0=(0,0))=1$$
and
$$p(0,y;\omega)=\pr_{(0,0)}^{\omega}(X_0=(0,y))=0$$ for all $y \neq 0$.
Now suppose the lemma is true for some $t$ and all $y \in \Z$. The transition probabilities for $X_{t+1}$ given in (\ref{ravegreen}) and (\ref{ravegreen2}) and the mass splitting rule in (\ref{simplifies}) show that for an arbitrary $y$
\begin{eqnarray*}
\pr_{(0,0)}^{\omega}(X_{t+1}=(t+1,y))
&=&\pr_{(0,0)}^{\omega}(X_{t}=(t,y+1))\frac{e^-(t,y+1)}{v(t,y+1)}\\
&&     +  \pr_{(0,0)}^{\omega}(X_{t}=(t,y-1))\frac{e^+(t,y-1)}{v(t,y-1)}\\
&=&p(t,y+1;\omega)\frac{e^-(t,y+1)}{v(t,y+1)}
     +  p(t,y-1;\omega)\frac{e^+(t,y-1)}{v(t,y-1)}\\
&=&p(t+1,y;\omega).
\end{eqnarray*}
Thus the induction hypothesis holds for $t+1$ and all $y \in \Z$. This completes the proof.
\end{proof}

This gives us a Markov chain $\hat{X}=(X_n)_{n \ge 0}$ for a fixed environment $\omega$. $\mathbb{P}^{\omega}$ is called the \textit{quenched law}. The \textit{joint annealed law} is given by
\begin{eqnarray*}
\tilde{\mathbb{P}}_{(t,y)}(d\hat{X},d\omega)=\Pomega_{(t,y)}(d\hat{X})\tilde{\mathbb{P}}(d\omega).
\end{eqnarray*}
The law $\tilde{\mathbb{P}}_{(t,y)}(d\hat{X},\Omega)$ is called the \textit{marginal annealed law} or just \textit{annealed law}. Define $\mathbf{S}^{(n)}(t)$ by
\begin{equation}
\frac{X_{[nt]}-([nt],0)}{\sqrt{n}}=\left(0,\mathbf{S}^{(n)}(t)\right).
\end{equation}
Since the major breakthrough in \cite{timo}, there has been a lot of work in RWRE. One of the major technical hurdles one has to overcome to prove QIP is to control the difference of two random walks driven by the same environment $\omega$. In \cite{timo}, the annealed law of this process is simply a random walk perturbed at the origin. In \cite{timoballistic} this process becomes significantly more complicated due to the backtracking nature of such walks although the environment is $i.i.d$. Later, \cite{mathew} applied the techniques of \cite{timoballistic} in an environment with spatial correlations present. Besides these models, a large variety of RWRE in dynamic correlated environments have been studied in recent years, see for example \cite{boldrighini2} \cite{boldrighini1} \cite{dolgopyat1} \cite{dolgopyat2} \cite{mathew} \cite{redig}. The methods used vary from Fourier series in \cite{boldrighini2} \cite{boldrighini1}, to martingale approximation in \cite{dolgopyat1} \cite{dolgopyat2}, to regeneration times in \cite{redig}. Also note that the a version of RWRE where the environment is defined by random walks  is considered in \cite{impa}.
 However to the best of our knowledge, a clear general idea about how to solve such problems seems to be missing as of now.\\\\
The major difficulties that arise in our model are the following:
\begin{enumerate}
\item[(i)] Our model is \textit{not symmetric}, even for fixed $\omega$, as the first co-ordinate (time) always takes a deterministic step of 1 with each transition of the Markov chain $\hat{X}$. This makes it different from the random conductance model, see \cite{biskup}. This environment thus falls in the category of \textit{random walk in random space-time environments} due to the deterministic increase in the time at each step.  
\item[(ii)] The environment under consideration has \textit{high correlations} between different points in space-time. The correlated environment models (such as \cite{boldrighini1} and \cite{mathew}) considered so far are rather specialized. As far as we know, the assumptions made in those works do not apply to our model.
\item[(iii)] Our model is \textit{not elliptic} in the sense of condition $(A5)$ of \cite{dolgopyat1} as there are points in space-time where the walker's next move is entirely determined by the environment. These points $(t,y)$ occur where all of the walkers in 
$\omega$ moved in the same direction ($e^+(t,y)>e^-(t,y)=0$ or vice versa). A sub-case of this that creates more trouble is the possibility of having a long space-time path $$\tau = \{(t,\tau(t)): t \in [T_1,T_2], v(t,\tau(t))=1 \ \text{and} \ |\tau(t+1)-\tau(t)|=1 \ \forall \ T_1 \le t \le T_2-1\}$$ in $\omega$. This is because, if the two copies of the random walk driven by $\omega$ get stuck in $\tau$, they are forced to move together for a long time resulting in some kind of `stickiness'. In Lemma \ref{lem:poisson}, we show that with high probability such paths cannot be stuck too long.

It is worth mentioning that in the $i.i.d$ environment setting of \cite{timo} a weaker form of ellipticity (see hypothesis (ME) there) is sufficient for their general technique of proving QIP to work. But with the high correlation in our environment, it is not clear how to directly apply these techniques. 
\end{enumerate}

These differences from the models in the literature present some technical challenges, although we adapt the same rough outline as in \cite{timo} to prove a quenched invariance principle.
Now we present the main theorem of this article:
\begin{thm}\label{thm:QCLT}
For almost all $\omega$ with respect to $\tilde{\mathbb{P}}$, the process $\mathbf{S}^{(n)}$ converges weakly to a standard Brownian motion $B$ in the usual Skorohod topology under $\Pomega_{(0,0)}$.
\end{thm}
From this we deduce the answer to the question that was posed at the beginning of this paper.

\begin{cor}
For a.e. $\omega$ with respect to $\tilde{\mathbb{P}}$, the mass distribution $\mu_t$ on $\mathbb{R}$ given by 
\begin{equation*}
\mu_t(y)=
\begin{cases}
p(t,\sqrt{t}y, \omega) & \mbox{if } y \in t^{-1/2}\mathbb{Z},\\
0 & \mbox{otherwise},
\end{cases}
\end{equation*}
converges weakly to the standard normal distribution as $t \rightarrow \infty$.
\end{cor}

\begin{proof}
Fix an environment $\omega$. As we have shown in Lemma \ref{equivalence}, $p(t,y;\omega)=\pr^{\omega}(X_t=(t,y))$ for all $t \in \N$ and all $y \in \Z$ so the scaled mass distribution $\mu_t$ at time $t$ is given by the distribution of $\mathbf{S}^{(n)}(t)$. The corollary follows directly from Theorem \ref{thm:QCLT}.
\end{proof}

Before we set off proving the theorem, we recall some of the machinery involved in proving \QCLT in general space-time random environments, mostly taken from \cite{timo}.

\section{Quenched invariance principle in general space-time random environments}
Let
\begin{eqnarray*}
D(\omega)=\Eomega_{(0,0)}(X_1)=\left(1,\frac{e^+(0,0)-e^-(0,0)}{v(0,0)}\right).
\end{eqnarray*}
For a bounded measurable function $h$ on $\Omega$ define the operator
\begin{eqnarray*}
\Pi h (\omega)=\frac{e^+(0,0)h(T_{(1,1)}\omega) + e^-(0,0)h(T_{(1,-1)}\omega)}{v(0,0)}.
\end{eqnarray*}
The operator $\displaystyle{h \mapsto \Pi h - h}$ defines the generator of the Markov process $\{T_{X_n}(\omega)\}_{n \ge 0}$ of the environment as `seen from the particle', i.e. keeping the particle fixed at the origin and shifting the environment around it. The transition probability of this Markov process is given by
 \begin{eqnarray*}
\pi(\omega ,A)=\Pomega_{(0,0)}(T_{X_1}(\omega) \in A).
\end{eqnarray*}
This Markov process, first conceived by \cite{papavaradhan} plays a key role in the arguments.\\\\
Now suppose that this Markov process has a stationary ergodic measure $\mathbb{P}^{\infty}$. Then we can extend the operator  $\Pi$ to a contraction on $L^p(\Omega,\mathbb{P}^{\infty})$ for any $p \in [1,\infty]$. Denote the joint law of $(\omega,T_{X_1}(\omega))$ by
\begin{eqnarray*}
\mu_2^{\infty}(d\omega_0,d\omega_1)=\pi(\omega_0,d\omega_1)\mathbb{P}^{\infty}(d\omega_0).
\end{eqnarray*}
For a fixed $\omega$, define the filtration $\left(\mathcal{F}^{\omega}_n\right)_{n\ge 0}$ where $\displaystyle{\mathcal{F}^{\omega}_n=\sigma(X_0,X_1,\cdots,X_n)}$. The main idea involved in proving \QCLT is to create a martingale $M$ with respect to $(\mathbb{P}^{\omega}_{(0,0)},\mathcal{F}^{\omega}_n)$ that is `close' to the process $X$ in some sense, and then to apply martingale central limit theorems to $M$ and translate the results to $X$. The main challenge involved is in the last step: to control the error involved in estimating $X$ by $M$. To prove convergence of $X$ to a Brownian motion in the Skorohod topology, we need to show that the error is uniformly small along a `typical' path of $X$. The way to do this, which we elaborate now, has its roots in the seminal work of Kipnis and Varadhan, and subsequently extended to the non-reversible set-up by Maxwell and Woodroofe \cite{maxwellwoodroofe} and Derriennic and Lin \cite{lin}.\\\\
Towards this end, note that the process
\begin{eqnarray*}
\overline{X}_n=X_n-\sum_{k=0}^{n-1} D\left(T_{X_k}(\omega)\right)
\end{eqnarray*}
is a $(\mathbb{P}^{\omega}_{(0,0)},\mathcal{F}^{\omega}_n)$ martingale. So, if we can prove that $\sum_{k=0}^{n-1} D\left(T_{X_k}(\omega)\right)-(n,0)$ can be expressed as a $(\mathbb{P}^{\omega}_{(0,0)},\mathcal{F}^{\omega}_n)$ martingale plus some error that can be controlled, a quenched CLT can be established for $(X_n)_{n \ge 0}$. We also mention here that to get a non-trivial quenched CLT, in the proof of Theorem~\ref{thm:QCLT} we need to show that the two martingales obtained as described above should not cancel each other (i.e. their sum should not be identically zero). Now, a standard way to have this decomposition is to investigate whether there exists a solution $h$ to \textit{Poisson's Equation}
\begin{eqnarray*}
h=\Pi h + D-(1,0).
\end{eqnarray*}
If such a solution $h$ did exist, we could write down
\begin{eqnarray*}
\sum_{k=0}^{n-1} D\left(T_{X_k}(\omega)\right)-(n,0)=\sum_{k=0}^{n-1}\left[h\left(T_{X_{k+1}}(\omega)\right)-\Pi h\left(T_{X_k}(\omega)\right)\right] + \left[h\left(T_{X_0}(\omega)\right)-\Pi h\left(T_{X_{n-1}}(\omega)\right)\right].
\end{eqnarray*} 

Note that the first term in the above expression is a $(\mathbb{P}^{\omega}_{(0,0)},\mathcal{F}^{\omega}_n)$ martingale and the second term is $L^2$ bounded with respect to $\mathbb{P}^{\infty} \otimes \mathbb{P}^{\omega}_{(0,0)}$. This would immediately give us the result. But unfortunately, a solution to the Poisson equation might not exist. To see this, note that if it did, then
\begin{eqnarray*}
\mathbb{E}^{\infty}\left\vert\Eomega_{(0,0)}( X_n-(n,0))\right\vert^2=\mathbb{E}^{\infty}\left\vert\Eomega_{(0,0)}h\left(T_{X_0}(\omega)\right)-\Eomega_{(0,0)}\Pi h\left(T_{X_{n-1}}(\omega)\right)\right\vert^2=O(1),
\end{eqnarray*}
but the above does not hold for the i.i.d.\ space-time product environment where $\mathbb{E}^{\infty}\left\vert\Eomega_{(0,0)} (X_n-(n,0))\right\vert^2$ grows like $\sqrt{n}$.
Rassoul-Agha and Sepp{\"a}l{\"a}inen \cite{timo} follow an extension of this idea, first introduced by \cite{kipnis}, and later extended by Toth \cite{toth} and subsequently Maxwell and Woodroofe \cite{maxwellwoodroofe}, to deal with the case where a solution to Poisson equation does not exist.

 Let $h_{\epsilon}$ be a solution of
\begin{eqnarray*}
(1+\epsilon)h=\Pi h + D-(1,0).
\end{eqnarray*}
Call $g=D-(1,0)$. Then, the solution of the above can be written as
\begin{eqnarray}\label{eqnarray:hepsilon}
h_{\epsilon}=\sum_{k=1}^{\infty}\frac{\Pi^{k-1}g}{(1+\epsilon)^k}=\epsilon\sum_{n=1}^{\infty}\frac{\sum_{k=0}^{n-1}\Pi^{k-1}g}{(1+\epsilon)^{n+1}},
\end{eqnarray}
which is in $L^2(\Omega,\mathbb{P}^{\infty})$. Define
\begin{eqnarray*}
H_{\epsilon}(\omega_0,\omega_1)=h_{\epsilon}(\omega_1)-\Pi h_{\epsilon}(\omega_0).
\end{eqnarray*}
Then, as our random walk has bounded range, Theorem $2$ of \cite{timo} says the following:

\begin{thm}\label{thm:SR}
Let $\mathbb{P}^{\infty}$ be any stationary ergodic probability measure for the Markov process on the environment $\Omega$ generated by $\Pi - I$. Also assume that there is $\alpha<1$ such that
\begin{eqnarray}\label{eqnarray:momentcond}
\mathbb{E}^{\infty}\left\vert\Eomega_{(0,0)}(X_n-(n,0))\right\vert^2=O(n^{\alpha}).
\end{eqnarray}
Then $H=\lim_{\epsilon \rightarrow 0}H_{\epsilon}$ exists in $L^2(\mu_2^{\infty})$ and for $\mathbb{P}^{\infty}$ a.e. $\omega$, $\mathbf{S}^{(n)}$ converges weakly to a Brownian motion with variance
\begin{eqnarray}\label{eqnarray:diffcoeff}
\sigma^2=\mathbb{E}^{\infty}\Eomega_{(0,0)}\left( X_1-D(\omega)+H(\omega,T_{X_1}(\omega))\right)^2.
\end{eqnarray}
\end{thm}

So, to prove Theorem \ref{thm:QCLT} with the aid of Theorem \ref{thm:SR}, we need to prove the existence of a stationary distribution $\mathbb{P}^{\infty}$ for the Markov process of the environment as seen by the particle (generated by $\Pi-I$), the moment bound (\ref{eqnarray:momentcond}) and the ergodicity of the Markov chain $\{T_{X_n}(\omega)\}$ with respect to $\mathbb{P}^{\infty}$. We prove these in the following section.

\section{Proof of Theorem \ref{thm:QCLT}}
\begin{pfofthm}{\ref{thm:QCLT}}
We break this proof up into three parts corresponding to the three conditions in Theorem \ref{thm:SR}. In Section \ref{subsection:stationarity}, we show the existence of a stationary measure $\mathbb{P}^{\infty}$ for the Markov process of the environment with Lemma \ref{lem:stationary}. In Section \ref{subsection:momentcondition}, we prove Lemma \ref{lem:momentbound} which gives a moment condition that verifies (\ref{eqnarray:momentcond}). In Section \ref{subsection:ergodicity}, we prove the ergodicity of the Markov process $\{T_{X_n}(\omega)\}_{n \ge 1}$ with respect to the stationary measure $\mathbb{P}^{\infty}$. This is in Lemma \ref{elevated}. These three lemmas verify the hypothesis of Theorem \ref{thm:SR} thus proving the desired almost sure invariance principle. 

We now prove that the \textit{annealed law} of $X_n$ is that of $(n,S_n)$, where $S_n$ is a simple random walk with steps $+1,-1$ with equal probability. This shows that $\sigma^2$ in Theorem \ref{thm:SR} is one for our case.

We show this when $X_n$ starts at $(0,0)$. The general case follows similarly. For any $i \in \Z$, let $\mathcal{F}_i$ denote the filtration
\begin{eqnarray}\label{eqnarray:sigmafield}
\mathcal{F}_i =\sigma \{e^+(t,y),e^-(t,y): \ t < i, \ y \in \mathbb{Z}\}.
\end{eqnarray}
Observe that
\begin{eqnarray}\label{eqnarray:bin}
e^+(i,y) | \mathcal{F}_i \sim \operatorname{Bin}\left(v(i,y),1/2\right)\nonumber\\
e^-(i,y) | \mathcal{F}_i \sim \operatorname{Bin}\left(v(i,y),1/2\right).
\end{eqnarray}
Then it follows from (\ref{ravegreen}) and \eqref{eqnarray:bin} that
\begin{align*}
\tilde{\mathbb{E}}\left(\Pomega_{(0,0)}(X_{n+1}=(t'+1,y'+1), X_n=(t',y')) \Big\lvert \mathcal{F}_i\right)&=\frac{1}{2}\Pomega_{(0,0)}(X_n=(t',y')).
\end{align*}
Taking expectation with respect to $\tilde{\mathbb{P}}$ on both sides, we see that the annealed law of $X_n$ is that of $(n,S_n)$.
\end{pfofthm} 


\subsection{Stationarity}\label{subsection:stationarity}
In Lemma \ref{lem:stationary}, we show that the measure $\tilde{\mathbb{P}}$ is in fact a stationary measure for the Markov process of the environment.
\begin{lem}\label{lem:stationary}
The measure $\tilde{\mathbb{P}}$ is stationary for the Markov process on $\Omega$ as seen from the particle, that is, the Markov process $\{T_{X_n}(\omega)\}_{n \ge 1}$.
\end{lem}
\begin{proof}
Let $\tilde{\mathbb{E}}$ denote the corresponding expectation. Take $\psi$ measurable with respect to $\mathcal{G}$. It will suffice to show that $\tilde{\mathbb{E}}(\Pi \psi)=\tilde{\mathbb{E}}(\psi)$. Note that
\begin{eqnarray*}
\tilde{\mathbb{E}}(\Pi \psi)&=&\mathbb{E}(v(0,0)\Pi \psi)\\
&=&\mathbb{E}\left(e^+(0,0)\psi(T_{(1,1)}\omega)\right)+ \mathbb{E}\left(e^-(0,0) \psi(T_{(1,-1)}\omega)\right)\\
&=&\mathbb{E}\left(e^+(-1,-1)\psi(\omega)\right)+ \mathbb{E}\left(e^-(-1,1) \psi(\omega)\right)\ \mbox{ (by translation invariance of $\mathbb{P}$)}\\
&=&\mathbb{E}\left(v(0,0)\psi(\omega)\right) \hspace{90pt} \mbox{ (as $e^+(-1,-1)+e^-(-1,1)=v(0,0)$)}\\
&=&\tilde{\mathbb{E}}(\psi).
\end{eqnarray*}
This proves the lemma.
\end{proof}


\subsection{Moment condition}\label{subsection:momentcondition}

In this section we verify the moment condition by proving the following lemma.

\begin{lem}\label{lem:momentbound}
$\tilde{\mathbb{E}}\left\vert\Eomega_{(0,0)}(X_n)-(n,0)\right\vert^2= O(\sqrt{n}(\log n)^3)$.
\end{lem}

\begin{sketchpfoflem}{\ref{lem:momentbound}}
Before giving the complete proof, we start by sketching the main ideas of the proof.

Our strategy is as follows. 
Fix an environment $\omega$.
Let $X$ and $\tilde{X}$ be two independent copies of the random walk starting from $(0,0)$ and running in the \textit{same environment} $\omega$. 
Also define $Y_i$ by $(0,Y_i)=X_i-\tilde{X}_i$ 
and let
 $\tilde{\mathbb{P}}^*$ be the law of the Markov process $$\left(X_i,\tilde{X}_i, T_{X_i},T_{\tilde{X}_i}\right).$$ 
Finally let $\tilde{\mathbb{E}}^*$ be its expectation.
 
 \begin{enumerate}
\item[\bf{Step 1}] \label{duwamish}  
Our first task is to establish
\begin{equation}\label{eqnarray:findY2}
\tilde{\mathbb{E}}\left\vert\Eomega_{(0,0)}(X_n)-(n,0)\right\vert^2
\leq \tilde{\mathbb{E}}^*\bigg(\#\{i:\ Y_i=0\}\bigg).
\end{equation}

\item[\bf{Step 2}] 
Next we show that $Y_i$ behaves like (two times) a lazy random walk when $Y_i \neq 0$.
But when $Y_i=0$ the distribution of $Y_{i+1}$ has a complicated dependence structure.

\item[\bf{Step 3}] 
We use the sequence $\{Y_i\}_{i=0}^{n}$ to divide the interval $[0,n]$ into excursions (maximal connected intervals where $Y_i \neq 0$) and holding times 
(maximal connected intervals where $Y_i = 0$). By the previous step we know the distribution of excursion lengths. Much of the work of this proof is in estimating the distribution of the lengths of the holding times.

\item[\bf{Step 4}] 
We bound the right hand side of (\ref{eqnarray:findY2}) by breaking it up into two parts. One is when the longest holding time is at most  $k\log^3(n)$ and the other is when the longest holding time is longer than $k\log^3(n)$.
Using the fact that excursions have the same distribution as a lazy random walk we are able to show that for a typical sequence $Y_i$ there are about $\sqrt{n}$ holds (and also about $\sqrt{n}$ excursions).  So the first part contributes approximately
$\sqrt{n}\log^3(n)$ to the expectation. We show that the probability of a long holding time is going to zero sufficiently quickly so the second part is contributing very little to the expectation (provided that $k$ is sufficiently large). Adding up these two bounds completes the proof.

\end{enumerate}

\end{sketchpfoflem}


\begin{lem}\label{fortythree}
\begin{equation}\label{eqnarray:findY3}
\tilde{\mathbb{E}}\left\vert\Eomega_{(0,0)}(X_n)-(n,0)\right\vert^2
\leq \sum_{i=0}^{n-1}\tilde{\mathbb{P}}^*(Y_i=0)
= \tilde{\mathbb{E}}^*\bigg(\#\{i \le n:\ Y_i=0\}\bigg).
\end{equation}

\end{lem}

\begin{proof}
It follows from the martingale decomposition of $X_n$ that
\begin{equation*}
\Eomega_{(0,0)}(X_n)=\sum_{i=0}^{n-1}\Eomega_{(0,0)} D(T_{X_i}(\omega)).
\end{equation*}
Recall that $g=D-(1,0)$. Therefore,
\begin{eqnarray*}
\tilde{\mathbb{E}}\left\vert\Eomega_{(0,0)}(X_n)-(n,0)\right\vert^2 &=&\tilde{\mathbb{E}}\left\vert\sum_{i=0}^{n-1}\Eomega_{(0,0)} g(T_{X_i}(\omega))\right\vert^2\\
&=& \tilde{\mathbb{E}}\left\vert\sum_{i=0}^{n-1}\sum_{y \in \mathbb{Z}}\Pomega_{(0,0)}(X_i=(i,y))g(T_{(i,y)}\omega)\right\vert^2\\
&=& \tilde{\mathbb{E}}\left[\sum_{i,j=0}^{n-1}\sum_{y,z \in \mathbb{Z}}\Pomega_{(0,0)}(X_i=(i,y))\Pomega_{(0,0)}(X_j=(j,z))g(T_{(i,y)}\omega).g(T_{(j,z)}\omega)\right],
\end{eqnarray*}
where in the above sum, we take a term to be zero if $\Pomega_{(0,0)}(X_i=(i,y))=0$ or $\Pomega_{(0,0)}(X_j=(j,z))=0$. Note that consequently, we only have terms with $v(i,y)\ge 1$ and $v(j,z) \ge 1$.\\\\
If $i>j$, we can condition on $\mathcal{F}_i$ defined in \eqref{eqnarray:sigmafield}
to get\\\\
$\tilde{\mathbb{E}}\left(\Pomega_{(0,0)}(X_i=(i,y))\Pomega_{(0,0)}(X_j=(j,z))g(T_{(i,y)}\omega).g(T_{(j,z)}\omega) \middle| \mathcal{F}_i\right)$\\\\
\hspace*{3cm}$=\Pomega_{(0,0)}(X_i=(i,y))\Pomega_{(0,0)}(X_i=(j,z))g(T_{(j,z)}\omega).\tilde{\mathbb{E}}\left(g(T_{(i,y)}\omega)\middle| \mathcal{F}_i\right)=0$\\\\
using \eqref{eqnarray:bin}.

Also, the terms in the above sum corresponding to $i=j$ and $y \neq z$ vanish as $(e^+(i,y),e^-(i,y))$ and $(e^+(i,z),e^-(i,z))$ are conditionally independent given $\mathcal{F}_i$ implying
\begin{eqnarray*}
\tilde{\mathbb{E}}\left(g(T_{(i,y)}\omega).g(T_{(i,z)}\omega) | \mathcal{F}_i\right)=\tilde{\mathbb{E}}\left(g(T_{(i,y)}\omega)| \mathcal{F}_i\right).\tilde{\mathbb{E}}\left(g(T_{(i,z)}\omega) | \mathcal{F}_i\right)=0.
\end{eqnarray*}
Thus the only terms in the above sum which do not vanish are those with $i=j$ and $y=z$. For these terms, we see that
\begin{eqnarray*}
\tilde{\mathbb{E}}\left(|g(T_{(i,y)}\omega) |^2 |\mathcal{F}_i\right)&=&\tilde{\mathbb{E}}\left(\frac{(e^+(i,y)-e^-(i,y))^2}{v^2(i,y)} | \mathcal{F}_i\right)\\
&=&4\tilde{\mathbb{E}}\left(\frac{(e^+(i,y)-\frac{1}{2}v(i,y))^2}{v^2(i,y)} | \mathcal{F}_i\right)=\frac{1}{v(i,y)}
\end{eqnarray*}
by (\ref{eqnarray:bin}). Thus we get
\begin{eqnarray}\label{eqnarray:findY}
\tilde{\mathbb{E}}\left\vert\Eomega_{(0,0)}(X_n)-(n,0)\right\vert^2&=&\tilde{\mathbb{E}}\left(\sum_{i=0}^{n-1}\sum_{y \in \mathbb{Z}}\left[\Pomega_{(0,0)}(X_i=(i,y))\right]^2\frac{1}{v(i,y)}\right)\nonumber\\
&\le & \tilde{\mathbb{E}}\left(\sum_{i=0}^{n-1}\sum_{y \in \mathbb{Z}}\left[\Pomega_{(0,0)}(X_i=(i,y))\right]^2\right)\nonumber\\
&=&\sum_{i=0}^{n-1}\tilde{\mathbb{P}}^*(Y_i=0).
\end{eqnarray}

\end{proof}

Also, notice that although the two walks $X$ and $\tilde{X}$ are conditionally independent given $\omega$, the annealed joint law of $(X,\tilde{X})$ is far from being the product of the annealed marginal laws of $X$ and $\tilde{X}$, as the subsequent calculations will indicate. This is what creates a major technical obstacle.\\\\

Now we start to understand the annealed law of the  process $Y_i$. We will show that it behaves like a lazy random walk provided that $Y_i \neq 0$, but at 0 it behaves very differently. The following lemma describes the behavior of $Y_i$ away from $0$.
\begin{lem}
If $r \neq 0$, then
\begin{eqnarray*}
\tilde{\mathbb{P}}^*(Y_{n+1}=s\mid Y_n=r,Y_{n-1}=r_{n-1},...,Y_1=r_1)
&=&\begin{cases}
\frac{1}{4} & \mbox{ if } |s-r|=2, \\
\frac{1}{2} & \mbox{ if } s=r.
\end{cases}
\end{eqnarray*}

\end{lem}

\begin{proof}
Note that for $r,s$ with $r \neq 0$ and $s-r=0,+2$ or $-2$, we get
\begin{eqnarray*}
\lefteqn{\tilde{\mathbb{P}}^*(Y_{n+1}=s,Y_n=r,Y_{n-1}=r_{n-1},...,Y_1=r_1)}\\
&=&\tilde{\mathbb{E}}\sum_{y_i-z_i=r_i,y-z=r}\Pomega_{(0,0)}(X_1=(1,y_1),..,X_n=(n,y))\Pomega_{(0,0)}(X_1=(1,z_1),..,X_n=(n,z))\\
&\quad &\hspace{3cm}\times \tilde{\mathbb{E}}\left[\sum_{y'-z'=s}\Pomega_{(n,y)}(X_1=(1,y'))\Pomega_{(n,z)}(X_1=(1,z'))| \mathcal{F}_n\right],
\end{eqnarray*}
where the filtration $\{\mathcal{F}_n\}$ is defined in (\ref{eqnarray:sigmafield}). By (\ref{eqnarray:bin}),
\begin{eqnarray*}
\tilde{\mathbb{E}}\left[\sum_{y'-z'=s}\Pomega_{(n,y)}(X_1=(1,y'))\Pomega_{(n,z)}(X_1=(1,z'))| \mathcal{F}_n\right]
=\begin{cases}
1/4 & \mbox{ when } |s-r|=2,\\
1/2 & \mbox{ when } s=r.
\end{cases}
\end{eqnarray*}
Thus we get
   \begin{eqnarray*}
\lefteqn{\tilde{\mathbb{P}}^*(Y_{n+1}=s,Y_n=r,Y_{n-1}=r_{n-1},...,Y_1=r_1)}\hspace{4cm}&&\\
&=&\begin{cases}
\frac{1}{4}\tilde{\mathbb{P}}^*(Y_n=r,Y_{n-1}=r_{n-1},...,Y_1=r_1) & \mbox{ if } |s-r|=2, \\\\
\frac{1}{2}\tilde{\mathbb{P}}^*(Y_n=r,Y_{n-1}=r_{n-1},...,Y_1=r_1) & \mbox{ if } s=r.
\end{cases}
\end{eqnarray*}
The claim follows.
\end{proof}
Thus we see that away from zero, the process $Y$ behaves as a homogeneous Markov process with the given transition probabilities. But at zero, things are not so nice, as we see in the next lemma. If $Y_n=0$ and $v(X_n)=1$ then $Y_{n+1}=0$ as well.
But the following lemma also shows that if $Y_n=0$ and $v(X_n)>1$ then $Y_{n+1} \neq 0$ with probability at least 1/4, a fact that will become useful later.

\begin{lem}\label{lem:walkatzero}
\begin{eqnarray*}
\lefteqn{\tilde{\mathbb{P}}^*(Y_{n+1}=0,Y_n=0,Y_{n-1}=r_{n-1},...,Y_1=r_1)}\\\\
&=&\tilde{\mathbb{E}}\sum_{y_i-z_i=r_i,y \in \mathbb{Z}}\Pomega_{(0,0)}(X_1=(1,y_1),..,X_n=(n,y))\Pomega_{(0,0)}(X_1=(1,z_1),..,X_n=(n,y))\\
&\quad &\hspace{3cm}\times\frac{1}{2}\left(1+\frac{1}{v(n,y)}\right).
\end{eqnarray*}
\end{lem}

\begin{proof}
The following simple calculation proves the lemma.\\
\begin{eqnarray*}
\lefteqn{\tilde{\mathbb{P}}^*(Y_{n+1}=0,Y_n=0,Y_{n-1}=r_{n-1},...,Y_1=r_1)}\\\\
&=&\tilde{\mathbb{E}}\sum_{y_i-z_i=r_i,y \in \mathbb{Z}}\Pomega_{(0,0)}(X_1=(1,y_1),..,X_n=(n,y))\Pomega_{(0,0)}(X_1=(1,z_1),..,X_n=(n,y))\\
&\quad &\hspace{3cm}\times\tilde{\mathbb{E}}\left[\frac{\left(e^+(n,y)\right)^2+\left(e^-(n,y)\right)^2}{v^2(n,y)}| \mathcal{F}_n\right]\\\\\\
&=&\tilde{\mathbb{E}}\sum_{y_i-z_i=r_i,y \in \mathbb{Z}}\Pomega_{(0,0)}(X_1=(1,y_1),..,X_n=(n,y))\Pomega_{(0,0)}(X_1=(1,z_1),..,X_n=(n,y))\\
&\quad &\hspace{3cm}\times\frac{1}{2}\left(1+\frac{1}{v(n,y)}\right).
\end{eqnarray*}
\end{proof}

Note that the presence of the term $\displaystyle{\left(1+\frac{1}{v(n,y)}\right)}$ in the summand makes this process depend on its entire past while making transitions from zero, thus destroying its Markov property. This is the striking difference from its analogue in the i.i.d.\ case studied by \cite{timo}, where the process $Y$ is a homogeneous Markov process perturbed at zero.\\\\
To deal with this problem, we decompose the path $\{Y_i\}$ into \textit{excursions} away from zero and \textit{holding times} at zero. Denote the $j$-th excursion by $e_j$. 
Let the time interval spanned by $e_j$ be $[\alpha_j,\beta_j]$, i.e., $|e_j(0)|=|Y_{\alpha_j}|=2$ and $\beta_j=\inf \{k>\alpha_j: Y_k=0\}$. Let the first holding time at $0$ be $\gamma_0=\alpha_1-1$ and let $\gamma_j=\alpha_{j+1}-\beta_j-1$ be the holding time at zero between $e_j$ and $e_{j+1}$. Define $$a(n)=\sup\{j\ge 1: \alpha_j \le n\}$$ with $a(n)=0$ if the above set is empty. With these defined, we can write down
\begin{eqnarray}\label{eqnarray:holdingexc}
\sum_{i=0}^{n-1}\mathbbm{1}(Y_i=0) &\le & \sum_{i=0}^{a(n)}\gamma_i \le \left(a(n)+1\right)\left(\sup_{j \le n}\gamma_j\right).
\end{eqnarray}
Let the time duration of the $j$-th excursion be denoted by $T_j=\beta_j-\alpha_j$.\\\\
Denote by $R$ the homogeneous random walk starting from zero with probabilities of increments $2,0,-2$ being $1/4,1/2,1/4$ respectively, and let $$R^*_n=\sup_{k \le n}R_k.$$ Denote the law of $R$ by $\mathbb{P}_R$ and the corresponding expectation by $\mathbb{E}_R$. Let $$A_n=\sup\{k \ge 0:T_1+...+T_k \le n\}.$$ It is easy to see that $T_1+...+T_k$ has the same distribution as the hitting time of level $2k$ by $R$. Thus $\{A_n: n \ge 1\}$ and $\{R^*_n/2: n \ge 1\}$ have the same distribution. Note that $a(n)\le A_n+1$. Therefore,
\begin{eqnarray}\label{eqnarray:excursion}
\tilde{\mathbb{E}}^*a(n) &\le & \tilde{\mathbb{E}}^*(A_n+1)\nonumber\\
& = & \mathbb{E}_R(R^*_n/2+1)\nonumber\\
&\le & C\sqrt{n}
\end{eqnarray}
for some constant $C<\infty$. The last step above follows from reflection principle arguments, see \cite{spitzer}.

The excursions away from zero have the same law as those of a lazy random walk $R$ so their distribution is well understood.  Our main challenge is to provide an upper bound on the supremum of the holding times $\gamma_j$. For a lazy random walk $R$, the holding times are i.i.d.\ geometric random variables and we can derive an upper bound on the probability that a holding time is at least $\log n$.  The $\gamma_j$ are far from i.i.d.\ and they depend on the entire past till that time. In spite of this dependence we will bound the probability that a holding time $\gamma_j$ is bigger than $(\log n)^k$ for some positive number $k$. Our bound will be uniform in $j$ and independent of all previous holding times. Using these bounds we will be able to bound the probability that the supremum of the holding times is large.

To do this we note by Lemma \ref{lem:walkatzero} that if $Y_i=0$ and $v(X_i)\ge 2$ then there is at least a probability of $1/4$ that $Y_{i+1}\neq 0$, independent of anything in the process or the environment up to time $i$.

 We now define a set of stopping times to indicate when these times occur. These stopping times will give us our bound on the holding times.

We now define some stopping times for the Markov process $\displaystyle{(X_i, \tilde{X}_i,T_{X_i},T_{\tilde{X}_i})}$. 
For each $j \ge 1$, define:
\begin{eqnarray*}
I_1^{(j)}&=&\beta_j \wedge n,\\
I_{i+1}^{(j)}&=&
\begin{cases}
\inf\{k > I_i^{(j)}: v(X_k) \ge 2\}\wedge n \mbox{ if } X_{I_i^{(j)}+1}=\tilde{X}_{I_i^{(j)}+1},\\
I_i^{(j)} \mbox{ otherwise.}
\end{cases}
\end{eqnarray*}
Thus we get a bi-indexed family $\{I_i^{(j)}: 1 \le i <\infty\}$ which is well defined with probability one. 
Note that for each $j$, the sequence indexed by $i$ eventually becomes constant.\\\\

For $j \ge 1$ with $\beta_j <n$, these stopping times indexed by $i$ represent the times $t \in [\beta_j,\alpha_{j+1}-1] \cap [0,n]$ (the portion of the holding time at zero between $e_j$ and $e_{j+1}$ lying in $[0,n]$) where there are at least two particles at $X_t$ ($=\tilde{X}_t$) and the sequence becomes constant if either time $n$ or time $\alpha_{j+1}-1$ is reached.\\\\
Call the corresponding stopped sigma fields $\{\mathcal{F}^*_{I_i^{(j)}}: 1 \le i <\infty\}$. Let $k_j=\sup\{i\ge 1:I_1^{(j)}<I_2^{(j)}<\cdots<I_i^{(j)}\}$ with $k_j=1$ for a constant sequence. With this notation, $\gamma_j=I_{k_j}^{(j)}-I_1^{(j)}$.\\\\
Control over the holding times $\gamma_j$ is obtained through the following two lemmas. 
\begin{lem}\label{lem:stopping}
\begin{eqnarray}
\tilde{\mathbb{P}}^*(k_j \ge k) \le \left(\frac{3}{4}\right)^{k-1}.
\end{eqnarray}
\end{lem}
\begin{proof}
Note that, by the strong Markov property applied at these stopping times,
\begin{eqnarray*}
\tilde{\mathbb{P}}^*(k_j \ge k)&=&\tilde{\mathbb{E}}^*\mathbbm{1}(I_1^{(j)}<I_2^{(j)}<\cdots <I_k^{(j)})\\
&\le &\tilde{\mathbb{E}}^*\mathbbm{1}(I_1^{(j)}<I_2^{(j)}<\cdots <I_{k-1}^{(j)}<n) \ \tilde{\mathbb{E}}^*\left(\mathbbm{1}\left(X_{I_{k-1}^{(j)}+1}=\tilde{X}_{I_{k-1}^{(j)}+1}\right)\middle|\mathcal{F}^*_{I_{k-1}^{(j)}}\right)\\
&=&\tilde{\mathbb{E}}^*\mathbbm{1}(I_1^{(j)}<I_2^{(j)}<\cdots <I_{k-1}^{(j)}<n) \ \tilde{\mathbb{E}}^*\left[\frac{e^+\left(X_{I_{k-1}^{(j)}}\right)^2+e^-\left(X_{I_{k-1}^{(j)}}\right)^2}{v\left(X_{I_{k-1}^{(j)}}\right)^2}\middle|T_{X_{I_{k-1}^{(j)}}}\right] \ \mbox{ (by (\ref{eqnarray:bin})) }\\ 
&=&\tilde{\mathbb{E}}^*\mathbbm{1}(I_1^{(j)}<I_2^{(j)}<\cdots <I_{k-1}^{(j)}<n).\frac{1}{2}\left(1+\frac{1}{v\left(X_{I_{k-1}^{(j)}}\right)}\right)\\
&\le &\frac{3}{4}\tilde{\mathbb{E}}^*\mathbbm{1}(I_1^{(j)}<I_2^{(j)}<\cdots <I_{k-1}^{(j)})\le \left(\frac{3}{4}\right)^{k-1},
\end{eqnarray*}
where the last step follows by induction.
\end{proof}

\begin{lem}\label{lem:poisson}
For sufficiently large $M>0$,
\begin{eqnarray}
\tilde{\mathbb{P}}^*\left(\bigcup_{0\le i,j\le n}\{I_{i+1}^{(j)}-I_i^{(j)}\ge M^2\log^2n\}\right)
\le \frac{2}{n^{M-1}}.
\end{eqnarray}
\end{lem}
\begin{proof}
For each $u \le n$, let
\begin{eqnarray*}
\tau_u=\inf\{l \ge u:v(X_l) \ge 2\}.
\end{eqnarray*}
Note that by the stationarity of $\{T_{X_n}\}$ under $\tilde{\mathbb{P}}^*$,
\begin{eqnarray*}
\tilde{\mathbb{P}}^*(\tau_u \ge M^2\log^2n)=\tilde{\mathbb{P}}^*(\tau_0 \ge M^2\log^2n).
\end{eqnarray*}
Under $\tilde{\mathbb{P}}^*$, $v(0,k)\sim$ Poi($1$) for $k\neq 0$ and $v(0,0)$ has a \textit{size-biased} Poi($1$) distribution. Let
\begin{eqnarray*}
N_t^+&=&\#\{(k,i): k\ge 1, S^{(2k,i)}(t)<0\},\\
N_t^-&=&\#\{(k,i): k\le -1, S^{(2k,i)}(t)>0\},
\end{eqnarray*}
where $\{S^{(k,i)}\}$ are the walkers defining the environment $\omega$ as discussed in Section \ref{section:description}.
It is easy to check that $N_t^+,N_t^- \sim$ Poi($\mu_t$), where
\begin{equation}\label{equation:poi}
\mu_t=\sum_{k=1}^{\infty}\mathbb{P}(S_t\ge 2k),
\end{equation}
where $S$ is a simple random walk starting from zero. It follows from (\ref{equation:poi}) that there is $0<C<\infty$ such that $\mu_t \ge C\sqrt{t}$ for every $t>0$.

Note that if both $N_t^+$ and $N_t^-$ are non-zero, then the process $X$ starting from $(0,0)$ must intersect at least one random walk $S^{(2k,\cdot)}$, $k \neq 0$ at or before time $t$, and at the time of intersection, say $l$, $v(X_l)\ge 2$. Thus,
\begin{eqnarray*}
\tilde{\mathbb{P}}^*(\tau_0 \ge t) &\le & \tilde{\mathbb{P}}^*(N_t^+=0 \mbox{ or }N_t^-=0)\\
&\le & 2e^{-\mu_t}\\
&\le & 2e^{-C\sqrt{t}}.
\end{eqnarray*}
Thus,
\begin{eqnarray*}
\tilde{\mathbb{P}}^*(\tau_0 \ge M^2\log^2n) \le 2e^{-CM\log n}=\frac{2}{n^{CM}}.
\end{eqnarray*}
Consequently,
\begin{eqnarray*}
\tilde{\mathbb{P}}^*\left(\bigcup_{0\le i,j\le n}\{I_{i+1}^{(j)}-I_i^{(j)}\ge M^2\log^2n\}\right)
&\le&
\tilde{\mathbb{P}}^*\left(\exists \ u \le n \mbox{ with }\tau_u \ge M^2\log^2n\right)\\
&\le &\frac{2}{n^{CM-1}}
\end{eqnarray*}
by a simple union bound. \qed\\\\
Lemmas \ref{lem:stopping} and \ref{lem:poisson} yield the following corollary:
\begin{cor}\label{cor:holding}
There exists a constant $0<C<\infty$ that does not depend on $n$ such that for sufficiently large $M>0$, we can choose $M'$ (depending on $M$) satisfying
\begin{eqnarray}
\tilde{\mathbb{P}}^*(\sup_{j \le n}\gamma_j \ge M'\log^3n)\le Cn^{-M}
\end{eqnarray}
for all $n$.
\end{cor}
\begin{proof}
For sufficiently large $M$ and any $0 \le j,k \le n$, we have:
\begin{eqnarray*}
\tilde{\mathbb{P}}^*(\gamma_j \ge M^2k\log^2n)&=&\tilde{\mathbb{P}}^*(\gamma_j \ge M^2k\log^2n, I_k^{(j)}-I_1^{(j)}<M^2k\log^2n)\\
&+&\tilde{\mathbb{P}}^*(\gamma_j \ge M^2k\log^2n, I_k^{(j)}-I_1^{(j)}\ge M^2k\log^2n)\\
&\le &\tilde{\mathbb{P}}^*\left(\gamma_j >I_k^{(j)}-I_1^{(j)}\right)+ \tilde{\mathbb{P}}^*\left(\sum_{i=0}^{k-1}\left(I_{i+1}^{(j)}-I_i^{(j)}\right)\ge M^2k\log^2n\right)\\
& \le & \tilde{\mathbb{P}}^*(k_j \ge k)+\tilde{\mathbb{P}}^*\left(\bigcup_{0\le i,j\le n}\{I_{i+1}^{(j)}-I_i^{(j)}\ge M^2\log^2n\}\right)\\
& \le &\left(\frac{3}{4}\right)^{k-1} + \frac{2}{n^{CM-1}},
\end{eqnarray*}
where the last step follows from Lemmas \ref{lem:stopping} and \ref{lem:poisson}.\\\\
The assertion then follows by taking $k=\frac{CM}{\log(4/3)}\log n$ and the union bound.
\end{proof}
Now, to prove Lemma \ref{lem:momentbound}, notice that by (\ref{eqnarray:holdingexc}),
\begin{eqnarray*}
\tilde{\mathbb{E}}^*\left(\sum_{i=0}^{n-1}\mathbbm{1}(Y_i=0)\right)&\le & \tilde{\mathbb{E}}^*\left((a(n)+1)\left(\sup_{j \le n}\gamma_j\right)\right)\\
&=&\tilde{\mathbb{E}}^*\left((a(n)+1)\left(\sup_{j \le n}\gamma_j\right)\mathbbm{1}(\sup_{j \le n}\gamma_j < M'\log^3n)\right)\\
&+&\tilde{\mathbb{E}}^*\left((a(n)+1)\left(\sup_{j \le n}\gamma_j\right)\mathbbm{1}(\sup_{j \le n}\gamma_j \ge M'\log^3n)\right)\\
&\le &(M'\log^3n) \tilde{\mathbb{E}}^*(a(n)+1)+n^2\tilde{\mathbb{P}}^*(\sup_{j \le n}\gamma_j \ge M'\log^3n)\\
&\le &C_1\sqrt{n}\log^3n + Cn^{-(M-2)}\hspace{2.5cm} \mbox{ (by (\ref{eqnarray:excursion}))}\\
&\le & C_2\sqrt{n}\log^3n,
\end{eqnarray*}
choosing $M$ sufficiently large.\\\\
This, together with (\ref{eqnarray:findY}), gives Lemma \ref{lem:momentbound}.
\end{proof}

\subsection{Ergodicity from the point of view of a tagged particle}\label{subsection:ergodicity}

In this section we prove the ergodicity of the shift from the point of view of a tagged particle. Note that the measure $\tilde{P}$ on $\Omega$ along with the transition probabilities $\pi(\omega,\cdot)$ define a unique measure $\mathbb{\boldsymbol{P}}$ on $\Omega^{\mathbb{Z}}$ which is invariant with respect to the shift $\theta: \Omega^{\mathbb{Z}} \rightarrow \Omega^{\mathbb{Z}}$ given by $$\left(\theta(\boldsymbol{\omega})\right)_i=\boldsymbol{\omega}_{i+1},$$ where $\boldsymbol{\omega}=(\omega_1,\omega_2,\dots) \in \Omega^{\mathbb{Z}}$. We say $\mathbb{\tilde{P}}$ is \textit{ergodic} for the Markov process $\{T_{X_n}(\omega)\}_{n \ge 0}$ if the measure $\mathbb{\boldsymbol{P}}$ is ergodic for $\theta$. See \cite{hairer2006ergodic} for more details.

 A measure preserving system $(Y,S,\nu)$ is said to be totally ergodic if $(Y,S^k,\nu)$ is ergodic for all $k$. Any measure preserving system that is strong mixing is also totally ergodic (see \cite{petersenbook}).
\begin{lem} \label{totallyergodic}
The vertical shift  $(\Omega,T_{(0,1)},\mathbb{P})$ is strong mixing and thus totally ergodic. In particular $(\Omega,T_{(0,2)},\mathbb{P})$ is ergodic. 
\end{lem}

\begin{proof}
For any cylinder set $R \subset \Omega$ which is defined by the values in a finite rectangle we have that for any sufficiently large $M$ the sets $T_{(0,M)}(R)$ and $R$ are independent. As these cylinder sets generate the $\sigma$-algebra, the shift $T_{(0,1)}$ is mixing and totally ergodic. Thus $T_{(0,2)}$ is ergodic. A (very slightly) more involved analysis shows that $T_{(0,1)}$ is weak Bernoulli and isomorphic to the shift on an infinite entropy i.i.d.\ measure. 
 \end{proof}
 
Now we want to define a new transformation built from $(\Omega,T_{(0,2)},\mathbb{P})$. 
This transformation is constructed from $(\Omega,T_{(0,2)},\mathbb{P})$ by inducing on the subset $\{v(0,0)>0\}$ and then building a ``Kakutani skyscraper" of height $v(0,0)$. The techniques of inducing and building skyscrapers are standard in ergodic theory and 
have been well studied in the theory of Kakutani equivalence \cite{orw} \cite{kakutani} \cite{petersenbook}. We start by describing the state space and the measure.
 
 Let $\Omegahat \subset \Omega \times \N$ consist of points of the form $(\omega,y)$ with $y\le v(0,0)$. 
 Let $\hatP$ be a measure on $\Omegahat$ defined as follows. (Technically $\hatP$ will be a measure on 
  $\Omega \times \N$ whose support is on $\Omegahat$.) Fix any $A \subset \Omega$ and $i \in \N$.
Then 
$$\hatP(A,i)=\mathbb{P}\left(A\cap \{i\le v(0,0)\}\right).$$
Let $\pi: \Omegahat \to \Omega$ be the projection map, $\pi(\omega,y)=\omega$. For any $A \subset \Omega$, let $\hat A \subset \Omegahat$ be defined by $\hat A=\pi^{-1}(A).$ 
Also, recall that $\tilde{\mathbb{P}}$ is an invariant measure for our process $\{T_{X_n}(\omega): n \in \N\}$. 
 These definitions give us the following lemma.

\begin{lem} \label{tunnel}
For any $A \subset \Omega$ $$\tilde{\mathbb{P}}(A)=\hatP(\hat A)$$ as both are the size biased version of the measure $\mathbb{P}$.
\end{lem}


Now we define the new transformation. For $\omega \in \Omega$, let 
$$n(\omega)=\inf\{m>0: 2\mid m,\ v(0,m)>0\} \ \ \ \ \text{ and } \ \ \ \ \tilde \omega =\T_{(0,n(\omega))}(\omega).$$
Next we define $\that:\Omegahat \to \Omegahat$ by
$$\that(\omega,y)=  \begin{cases} (\omega,y+1) &\mbox{if } y< v(0,0), \\
(\tilde \omega,1) & \mbox{if } y=v(0,0). \end{cases}$$

\begin{lem}
The transformation $(\hat \Omega,\that,\hatP)$ is measure preserving and ergodic.
\end{lem}

\begin{proof}
That the transformation is measure preserving is standard in ergodic theory (see Chapter 2.3 of \cite{petersenbook}) so we only check ergodicity. 
Suppose $A$ is an invariant set for $\that$ (that is, $\hatP(\that^{-1}A \ \Delta \ A)=0$). Set 
$$A'=\left[\bigcup_{m=-\infty}^\infty T^m_{(0,2)}(\pi(A))\right] \cap \{v(0,0)=0\}.$$
 If $\omega \in \pi(A) \cup A'$ then either $v(0,0)>0$ and
 there exists $y$ such that $(\omega,y) \in A$ or $v(0,0)=0$ and
 there exists $\omega'$ and $y'$ and $m'$ such that $(\omega',y') \in A$ and
 $T_{(0,2)}^{m'}(\omega')=\omega.$ It is easy to check that in either case we have 
 $T_{(0,2)}\omega \in \pi(A) \cup A'$ and 
 $\pi(A) \cup A'$ is invariant.

Also note that  as every point in $\pi^{-1}( \pi(A))$ is in the $\hat{T}$-orbit of a point in $A$ we have 
\begin{equation}\label{musselman}
A \subset \pi^{-1}( \pi(A)) \subset \bigcup_{m=-\infty}^{\infty}\that^m(A)=A,
\end{equation}
where the last containment is by the invariance of $A$. 

 Now assume that $\hatP(A)>0$. Then $\pr(\pi(A))>0$. So by the ergodicity of $T_{(0,2)}$ 
and the invariance of $\pi(A) \cup A'$ we have that $$\pr(\pi(A) \cup A')=1.$$ Then 
 $$\pr\{v(0,0)>0\}=\pr(\pi(A))$$ and by (\ref{musselman}) $$1=\hatP\{\pi^{-1}(\pi(A))\} = \hatP(A),$$
 so $\hatP(A)=1$ and $\that$ is ergodic.
\end{proof}



In the following two lemmas, we show that ergodicity of $(\hat \Omega,\that,\hatP)$ implies the ergodicity of $\tilde{\mathbb{P}}$ for the Markov process $\{T_{X_n}(\omega)\}_{n \ge 0}$.

 We call a set $A \subset \Omega$ \textit{invariant} for $T_{X_1}$ if $\mathbbm{1}_A(\omega)=\Pomega_{(0,0)}(T_{X_1}(\omega) \in A)$ for $\tilde{\mathbb{P}}\ a.e. \ \omega$.

\begin{lem} \label{surface}
If $A \subset \Omega$ is invariant for $T_{X_1}$, then $\hat A$ is invariant for $\that$.
\end{lem}

\begin{proof}
Let $A$ be an invariant set for $T_{X_1}$ with $\tilde{\mathbb{P}}(A)>0$. Then for every $n$, $\Pomega_{(0,0)}(T_{X_n} (\omega) \in A)=1$ for $\tilde{\mathbb{P}}$ a.e.\ $\omega \in A$ and $\Pomega_{(0,0)}(T_{X_n}(\omega) \in A)=0$ for $\tilde{\mathbb{P}}$ a.e.\ $\omega \not\in A$. Note that for $\tilde{\mathbb{P}}$ a.e. $\omega \in A$ with $v(0,0)>0$, there exist paths $\{\gamma_1(k) \in \mathbb{Z}: k \le N\}$ and $\{\gamma_2(k)\in \mathbb{Z}: k \le N\}$, with $\gamma_1,\gamma_2,N$ depending on $\omega$, such that $\gamma_1(0)=0$, $\gamma_2(0)=n(\omega)$, $\gamma_1(N)=\gamma_2(N)$ and $$\displaystyle{e^{\sgn(\gamma_i(j+1)-\gamma_i(j))}(j,\gamma_i(j))>0}$$ for $i=1,2$, $j \le N-1$. Let $X$ and $\hat{X}$ denote two random walks on the same environment $\omega$ starting from $(0,0)$ and $(0,n(\omega))$ respectively. This gives, for a given such $\omega$, a natural coupling $\lambda$ between the laws of $(T_{X_n}(\omega))$ and $(T_{X_n}(\tilde{\omega}))$ such that $$\lambda \left(T_{X_N}(\omega)=T_{\hat{X}_N}(\tilde{\omega})\right)>0.$$ Thus, a.s. $\omega \in A$, $\mathbb{P}^{\tilde{\omega}}_{(0,0)}(T_{X_N}(\tilde{\omega}) \in A)>0$ implying $\tilde{\omega}\in A$. Now by the definition of $\hat A$ and $\that$ we have that $\hat A$ is a.s.\ invariant under $\that.$
\end{proof}

\begin{lem} \label{elevated}
$\tilde{\mathbb{P}}$ is ergodic for the Markov process $\left(T_{X_n}(\omega)\right)_{n\ge0}$.
\end{lem}

\begin{proof}
Let $A \subset \Omega$ be an invariant set for $T_{X_1}$.
By Lemma \ref{surface}, $\hat A$ is invariant for $\that$. As $(\Omegahat,\that,\hatP)$ is ergodic we have $\hatP(\hat A)=0 \text{ or } 1$. By Lemma \ref{tunnel}, $\tilde{\mathbb{P}}( A)$ 
is also equal to  0 or 1. The ergodicity of $\tilde{\mathbb{P}}$ for the Markov process $\left(T_{X_n}(\omega)\right)_{n\ge0}$ now follows from Corollary 5.11, Pg. 42 of \cite{hairer2006ergodic}.
\end{proof}


\textbf{Conclusion and Remarks: }In higher dimensions, although estimating the number of intersections of the two independent random walks in the same environment is a more subtle issue, the developments in \cite{peres} indicate that results similar to Lemma \ref{lem:poisson} and Lemma \ref{surface} hold. Applying the rest of the techniques of this paper to the case of higher dimensions seems straightforward.

In addition to the technique of \cite{timo} that we adapted to our problem in this paper, there are methods to prove quenched CLT (among others) in \cite{sznit} which have been developed further by \cite{bergerquenched} (although these assume i.i.d. environments which are uniformly elliptic, or ballistic in a certain sense). These methods have the advantage of avoiding issues of ergodicity, but it remains to be verified whether they extend sufficiently to be useful in our problem. 

We hope to address these issues in a later article.

\textbf{Acknowledgements: }We wish to thank Krzysztof Burdzy for suggesting a similar problem which led to this article, and providing guidance throughout this project. S.B.\ gratefully acknowledges support from NSF grant number DMS-1206276.
C.H.\ gratefully acknowledges support from the NSF and the NSA (grants number DMS-1308645 and H98230-13-1-0827).

We also wish to thank two anonymous referees for their very useful suggestions and careful reading of the manuscript. In particular, we thank one of them for bringing to our notice the issues discussed in the conclusion.

\bibliographystyle{plain}
\bibliography{RMSbib}

\end{document}